\documentclass[a4,12pt]{article}
\usepackage{graphicx}   
\usepackage{amsmath,amsfonts,amssymb,amsthm}
\usepackage[margin=2cm]{geometry}
\usepackage{txfonts}%
\usepackage{bm}
\usepackage{bm}
\usepackage{comment}
\usepackage{cases}
\usepackage{color}

\theoremstyle{plain}
\newtheorem{theorem}{Theorem}[section]
\newtheorem{lemma}[theorem]{Lemma}
\newtheorem{corollary}[theorem]{Corollary}
\newtheorem{proposition}[theorem]{Proposition}
\newtheorem{assumption}[theorem]{Assumption}

\theoremstyle{definition}
\newtheorem{definition}[theorem]{Definition}

\theoremstyle{remark}
\newtheorem{remark}[theorem]{Remark}

\renewcommand{\indent}{\hspace*{5mm}}
\newcommand{\cK}{{\cal K}}

\newcommand{\cL}{{\cal L}}

\newcommand{\cps}{{\cal K}} 
\renewcommand{\ss}{{A}}
\newcommand{\epsilonr}{\frac{1}{\epsilon}}
\newcommand{\lambdar}{\frac{1}{\lambda}}
\newcommand{\Pizero}{\Pi_{0}}
\newcommand{\Psiinfty}{\Psi_{\infty}}
\newcommand{\simzero}{\sim_{0}}
\newcommand{\siminfty}{\sim_{\infty}}

\begin{document}
	
\title{Relation between the rate of convergence of strong law of large numbers and the rate of concentration of Bayesian prior in game-theoretic probability}

\author{
	Ryosuke Sato\thanks{Graduate School of Information Science and Technology, University of Tokyo}, \ 
	Kenshi Miyabe\thanks{School of Science and Technology, Meiji University} \ 
	and Akimichi Takemura\thanks{The Center for Data Science Education and Research, Shiga University}
}
\date{April, 2016}
\maketitle

\begin{abstract}
We study the behavior of the capital process of a continuous Bayesian mixture of fixed proportion betting strategies in the one-sided unbounded forecasting game 
in game-theoretic probability.    We establish the relation between the rate of convergence of the strong law of large numbers in the self-normalized form and the rate of divergence to infinity of the  prior density around the origin.  
In particular we present prior densities
ensuring the validity of Erd\H{o}s--Feller--Kolmogorov--Petrowsky law of the iterated logarithm.  
\end{abstract}


\section{Introduction}

The most basic proof of the strong law of large numbers (SLLN) in Chapter 3 of
\cite{ShaferVovk2001Probability} uses a discrete mixture
\begin{align}\label{eq:discrete-mixture}
\sum_{j=1}^\infty2^{-j-1}\prod_{i=1}^n(1+2^{-j}x_i) \; + \; 
\sum_{j=1}^\infty2^{-j-1}\prod_{i=1}^n(1-2^{-j}x_i)
\end{align}
of fixed proportion betting strategies.
The mixture puts the weight $2^{-j-1}$ on the points $\pm 2^{-j}$, $j=1,2,\dots$.
The sum of weights on the interval $[0,2^{-k}]$, $k\ge 1$,  is expressed  as
\[
\sum_{j=k}^\infty 2^{-j-1} = 2^{-k},
\]
which is equal to the length of the interval $[0,1/2^k]$.
Hence this mixture can be understood as the discrete approximation to the continuous uniform distribution 
over the interval $[-1/2, 1/2]$. 
In Chapter 3 of \cite{ShaferVovk2001Probability} this mixture is used only to prove
the usual form of SLLN, and any discrete distribution having the origin as an
accumulation point of mixture weights serves the same purpose. 
As we show in Section 
\ref{sec:rate}, the concentration of the mixture weights around the origin has
the direct implication on the rate of convergence of SLLN forced by the mixture.

In this paper we consider a continuous mixture in the integral form. 
Although a proof based on a discrete mixture is conceptually simpler, 
the integral in a continuous mixture is more convenient for analytic treatment.  
In fact in this paper 
we give a unified treatment covering the usual SLLN, the validity (i.e., the upper bound)
of the usual law of iterated logarithm (LIL), and finally  the validity
of Erd\H{o}s--Feller--Kolmogorov--Petrowsky (EFKP) LIL
\cite[Chapter 5.2]{Revesz2013Random}.

Another feature of this paper is the self-normalized form
in which we consider SLLN.
Let 
\[
S_n = X_1 + \dots + X_n,  \qquad \ss_n^2=X_1^2 + \dots + X_n^2
\]
denote the sum of $n$ random variables and
and the sum of their squares, respectively.  We compare 
$S_n$ to $\ss_n^2$, rather than to $n$. 
For example $S_n /n \rightarrow 0$ is the usual non-self-normalized SLLN, whereas
$S_n / \ss_n^2 \rightarrow 0$ is the SLLN in the self-normalized form.
We often obtain cleaner statements and proofs in the self-normalized form.
For instance, in measure-theoretic probability,
Griffin and Kuelbs \cite{GriffinKuelbs1991Some} showed the EFKP-LIL in the self-normalized form in i.i.d.\ case with some additional conditions.
Later, Wang \cite{WangSPL1999} and Cs{\"o}rg{\H{o}} et al. \cite{CsorgoAP2003} eliminated 
some of the conditions. A similar argument in the self-normalized form
can also be seen in game-theoretic probability
as shown in  \cite{sasai2015efkp-sn}.
Self-normalized processes including $t$-statistic are known for their statistical applications.
In studying self-normalized processes, self-normalized sums $S_n / \ss_n$ are regarded as
important and one of the reasons is that they have close 
relations to $t$-statistic; see, e.g., 
\cite{Efron1969}, \cite{Griffin-Phillip-ECP2002}.
For a survey of results concerning self-normalized processes, 
especially self-normalized sums, see \cite{Pena-Lai-Shao-2009book}, 
\cite{Shao-Wang-QlyingPS2013}, and references in \cite{sasai2015efkp-sn}.
In the self-normalized form of SLLN we only consider paths of Reality's moves such that
$\ss_\infty=\lim_n \ss_n=\infty$. Hence the events considered in this
paper are conditional events given $\ss_\infty=\infty$.

Yet another feature of this paper is that we mainly consider the one-sided unbounded forecasting game.
We notice that
the results on the bounded forecasting game can be derived from the results on  
the one-sided unbounded forecasting game.  Also it is an interesting fact that
in the one-sided unbounded forecasting game the usual non-self-normalized SLLN does not hold.
We discuss this fact in 
Section \ref{sec:other-properties}.
The importance of this game will be explained in Section \ref{sec:one-sided}
and Section \ref{sec:discussions}.

We remark on the terminology used in this paper.  We use ``increasing'' and ``decreasing'' in the weak sense, i.e.\, we simply write ``increasing'' instead of more accurate ``(monotone) nondecreasing''.  

The organization of this paper is as follows.  
In Section \ref{sec:one-sided} we introduce the protocol of the one-sided unbounded 
forecasting game and relate it to betting on positive price processes.  We also
define Bayesian strategies with prior densities.
In Section \ref{sec:preliminary} we establish some preliminary results on 
the capital process of a Bayesian strategy.
In Section \ref{sec:rate} we prove a general inequality for the capital process
of a Bayesian strategy and apply it to prove SLLN and the validity of LIL in the
self-normalized form based on Bayesian strategies.
In Section \ref{sec:equivalence} we define two functionals which connect prior densities
and functions of the upper class for EFKP-LIL.
Based on these results,
we give a proof of the validity of EFKP-LIL using a corresponding
Bayesian strategy.
In Section \ref{sec:other-properties} we prove other basic properties
of the one-sided unbounded forecasting game, including the rate of SLLN
in the usual non-self-normalized form and 
the Reality's deterministic strategy against the usual form of SLLN.
We end the paper with some discussions in Section \ref{sec:discussions}.

\section{The one-sided unbounded forecasting game and the Bayesian strategy}
\label{sec:one-sided}
In this paper we mainly consider the one-sided unbounded forecasting game
defined as follows.

\begin{quote}
	{\sc The One-sided Unbounded Forecasting Game} \quad (OUFG)\\
	\textbf{Players}: Skeptic, Reality\\
	\textbf{Protocol}:\\
	\indent Skeptic starts with the initial capital $\cps_0>0$.\\
	\indent FOR $n=1,2,\ldots$:\\
	\indent\indent Skeptic announces $M_n\in\mathbb{R}$.\\
	\indent\indent Reality announces $x_n\in[-1,\infty).$\\
	\indent\indent $\cps_n=\cps_{n-1}+M_n x_n$.\\
	\textbf{Collateral Duties}:
	Skeptic must keep $\cps_n$ nonnegative.
	Reality must keep $\cps_n$ from tending to infinity.
\end{quote}
In game-theoretic probability the initial capital $\cps_0$ is 
usually standardized to be 1.  However the value of the initial capital is not relevant
in discussing whether Skeptic can force an event $E$ or not. 
Also notation is sometimes simpler if we allow arbitrary positive value for $\cps_0$.

In OUFG, Reality's move $x_n$ is unbounded on the positive side.
On the other hand, in the bounded forecasting game (BFG) of Chapter 3 of \cite{ShaferVovk2001Probability}, $x_n$ is restricted as $|x_n|\le 1$.
Since the Reality's move space is smaller in BFG than in OUFG, Reality is weaker against Skeptic in BFG than in OUFG.  This implies that if Skeptic can force an event $E$ in OUFG, 
then he can force $E$ in BFG.
In this sense the  one-sided version of SLLN (cf.\ (3.9) in Lemma 3.3 of 
\cite{ShaferVovk2001Probability}) in OUFG is stronger than that in BFG.  
This is one of the reasons why we mainly consider OUFG in this paper.

Let $\epsilon_n = M_n /\cps_{n-1}$ denote the proportion of the capital Skeptic bets 
on the round $n$. The collateral duty of Skeptic is $\epsilon_n \in [0,1]$ in OUFG, whereas
it is $\epsilon_n \in [-1,1]$ in BFG.  
This again reflects the fact that Skeptic is stronger in
BFG than in OUFG.

OUFG is a natural protocol, when we consider betting on a positive price process $\{p_n\}$
of some risky financial assets.  Let
\[
x_n = \frac{p_n}{p_{n-1}} -1
\]
denote the return of the price process.  Then $x_n \ge -1$ and there is no upper bound for
the return $x_n$ if $p_n$ is allowed to take an arbitrary large value.  
Furthermore 
we can write $\cps_n=\cps_{n-1}+M_n x_n=\cps_{n-1}(1+\epsilon_n x_n)$ as
\begin{equation}
\label{eq:portfolio}
\cps_n=\cps_{n-1}\left(1+ \epsilon_n\left(\frac{p_n}{p_{n-1}} -1\right)\right)=
\cps_{n-1}\left((1-\epsilon_n) +  \epsilon_n \frac{p_n}{p_{n-1}} \right).
\end{equation}
Hence $(1-\epsilon_n)$ is the proportion of the current capital Skeptic keeps as cash, which 
does not change the value from round to round, and $\epsilon_n$ is 
the proportion of the current capital Skeptic bets on the risky asset.

When $\epsilon_n \equiv \epsilon$ is a constant, we call the strategy 
the constant-proportion betting strategy or the $\epsilon$-strategy.
The capital process of the $\epsilon$-strategy
is written as
\[
\cps_n = \cps_0 \prod_{i=1}^n (1+\epsilon x_i).
\]
We now consider continuous mixture of $\epsilon$-strategies where the initial capital
is distributed according to the (unnormalized)  prior density $\pi\ge 0$ of an absolutely continuous finite measure on the  unit interval $[0,1]$: 
\[
\int_0^1 \pi(\epsilon) d\epsilon < \infty.
\]
The betting proportion $\epsilon_n$ of the  Bayesian strategy with the prior density $\pi$ is defined as
\begin{equation}
\label{eq:eps-n}
\epsilon_n = \frac{\int_0^1 \epsilon  \prod_{i=1}^{n-1} (1+\epsilon x_i) \pi(\epsilon)d\epsilon}{\int_0^1 \prod_{i=1}^{n-1} (1+\epsilon x_i) \pi(\epsilon)d\epsilon}.
\end{equation}
Then 
the capital process $\cps_n^{\pi}$ of this strategy is written as
\begin{equation}
\label{eq:continuous-mixture}
\cps_n^{\pi} = \int_0^1 \prod_{i=1}^n 
(1+\epsilon x_i) \pi(\epsilon)d\epsilon
\end{equation}
with the initial capital $\cps_0^{\pi} = \int_0^1  \pi(\epsilon)d\epsilon$.
The equation \eqref{eq:continuous-mixture} is checked by induction on $n$, since
with $\epsilon_n$ given in \eqref{eq:eps-n}, we have
\begin{align*}
\cps_{n-1}(1+\epsilon_n x_n)&=
\int_0^1 \prod_{i=1}^{n-1} (1+\epsilon x_i) \pi(\epsilon)d\epsilon \;
\left(1 +  
\frac{\int_0^1 \epsilon  \prod_{i=1}^{n-1} (1+\epsilon x_i) \pi(\epsilon)d\epsilon}{\int_0^1 \prod_{i=1}^{n-1} (1+\epsilon x_i) \pi(\epsilon)d\epsilon} x_n \right) \\
&=\int_0^1 (1+\epsilon x_n)\prod_{i=1}^{n-1} (1+\epsilon x_i) \pi(\epsilon)d\epsilon 
=\int_0^1 \prod_{i=1}^{n} (1+\epsilon x_i) \pi(\epsilon)d\epsilon.
\end{align*}
A similar argument is also in \cite{sasai2015efkp-sn}.
One should notice that the equation \eqref{eq:discrete-mixture}
can be seen as one discretization of \eqref{eq:continuous-mixture}, where
$\pi$ is the uniform density on $[-1/2,1/2]$.

We are concerned with the rate of increase of $\pi(\epsilon)$ as $\epsilon \downarrow 0$.
Hence we require some regularity conditions on the behavior of $\pi(\epsilon)$ around the origin.
The following condition on $\pi$ is convenient in discussing EFKP-LIL in Section
\ref{sec:equivalence}.



\begin{assumption}
\label{ass:1}
There exist $\epsilon_\pi \in (0,1)$ and $\delta_\pi>0$ such that 
\begin{enumerate}
\item $\pi(\epsilon)$ is a prior density, that is, nonnegative and integrable on $[0,1]$,
\item $\pi(\epsilon)\ge \delta_\pi$ on $(0,\epsilon_\pi)$, and 
\item $\epsilon \pi(\epsilon)$ is 
increasing on
$(0,\epsilon_\pi)$.
\end{enumerate}
\end{assumption}

For simplicity, we allow the case that the integral of $\pi$ on $[0,1]$ is not $1$,
which does not cause a problem when considering limit theorems.
Note that we are allowing  $\infty = \lim_{\epsilon \downarrow 0} \pi(\epsilon)$, 
but by the monotonicity and the integrability we are assuming 
$0=\lim_{\epsilon\downarrow 0}\epsilon \pi(\epsilon)$.
Assumption \ref{ass:1} holds for particular examples discussed in the next section.
Concerning the condition on $\delta_\pi$ in Assumption \ref{ass:1}, 
Skeptic can always allocate the initial capital of $\delta_\pi$ to the uniform prior 
of Section \ref{subsec:uniform} to satisfy this condition.

\section{Preliminary results}
\label{sec:preliminary}

In this section we see the self-normalized form of SLLN
in the one-sided unbounded forecasting game.

\begin{proposition}
\label{prop:SLLN-1}
In OUFG, by any Bayesian strategy with $\pi$ satisfying Assumption \ref{ass:1}, Skeptic weakly forces
\begin{equation}
\label{eq:SLLN-1}
\ss_\infty=\infty \ \ \Rightarrow \ \ \limsup_{n} \frac{S_n}{\ss_n^2} \le 0.
\end{equation}
\end{proposition}

We can prove this proposition by more or less the same way
as in Lemma 3.3 of \cite{ShaferVovk2001Probability}.
The difference is that the mixture we use here is not discrete but continuous.
We will see later that we can show stronger forms in similar arguments.

We begin with the following simple lemma.  The particular case $C=1$ is used in
Lemma 3.3 of \cite{ShaferVovk2001Probability}.
We prove this lemma separately in a stronger form for later use.

\begin{lemma}
\label{lem:log-inequality}
For any $C>0$
\begin{equation}
\label{eq:log-inequality}
\ln(1+t) \ge t- \frac{1+C}{2} t^2  \qquad \text{\rm for} \quad t\ge -\frac{C}{1+C}.
\end{equation}
\end{lemma}

\begin{proof} Let
\[
g(t)=\ln(1+t) - t + \frac{1+C}{2} t^2, \quad t> -1.
\]
Then $g(0)=0$ and
\[
g'(t)= \frac{1}{1+t} - 1 + (1+C)t = \frac{t(c + (1+c)t)}{1+t}.
\]
Hence $g'(t)=0$ for $t=0$ or $t=-C/(1+C)$, and $g'(t)<  0$ only for
$-C/(1+C)<  t<0$.  This implies the lemma.
\end{proof}



\begin{proof}[Proof of Proposition \ref{prop:SLLN-1}]
We use Lemma \ref{lem:log-inequality} with $C=1$.
Suppose that Reality has chosen a path such that $\ss_\infty=\infty$ and
$\limsup_n S_n/\ss_n^2 > 0$. Then there exists a sufficiently small $0< \delta < \min(\epsilon_\pi,1/2)$ 
such that 
$S_n > \delta \ss_n^2$ for infinitely many $n$.
For such an $n$,
\begin{align*}
\cps_n^{\pi} 
&\ge \int_{\delta/3}^{2\delta/3}
\exp\left(\sum_{i=1}^n \ln(1+\epsilon x_i)\right) \pi(\epsilon)d\epsilon \\
& \ge \int_{\delta/3}^{2\delta/3}
\exp(\epsilon S_n - \epsilon^2 \ss_n^2) \frac{1}{\epsilon}\epsilon\pi(\epsilon) d\epsilon\\
&\ge \frac{3}{2\delta} \frac{\delta}{3} \pi\Big(\frac{\delta}{3}\Big)
\int_{\delta/3}^{2\delta/3} \exp(\ss_n^2 \epsilon (\delta - \epsilon)) d\epsilon \\
&\ge \frac{1}{2} \pi\Big(\frac{\delta}{3}\Big) \frac{\delta}{3}\exp\left(\ss_n^2 \frac{\delta^2}{9}\right).
\end{align*}
Notice that the strategy and the capital process $\cps_n^\pi$ do not depend on $\delta$.
Because  $\lim_n \ss_n = \infty$, we have
$\limsup_n \cps_n^{\pi}  = \infty$.
\end{proof}

\begin{remark}
Skeptic cannot force the event
$A_\infty=\infty\Rightarrow\liminf_{n}\dfrac{S_n}{\ss_n^2}\ge0$.
Reality can announce $x_n=-1$ for all $n$, and
$A_\infty=\infty$ and $\dfrac{S_n}{\ss_n^2}=-1$ for all $n$.
\end{remark}

\section{Rate of convergence of SLLN implied by a Bayesian strategy}
\label{sec:rate}

In this section we present results on the rate of convergence of SLLN by a Bayesian strategy, 
by establishing  some lower bounds for the capital process. 
Note that since the Bayesian strategy 
already weakly forces  \eqref{eq:SLLN-1}, we only need to consider lower bounds when 
$S_n /\ss_n^2$ is sufficiently small.

In Section \ref{subsec:basic-inequality} we present our first inequality, which will be 
used to show the rate of convergence of SLLN for the uniform prior in Section \ref{subsec:uniform}.
Then we give a more refined inequality in Section \ref{subsec:another-bound}, which
will be used to show the rate of convergence of SLLN 
for the power prior (Section \ref{subsec:power-prior}),
for the prior ensuring the validity 
in the usual form of LIL (Section \ref{subsec:prior-for-usual-lil})
and for the prior ensuring 
the validity of a typical form  of EFKP-LIL (Section \ref{subsec:prior-for-typical-EFKP-LIL}).

\subsection{A lower bound for the capital process of a Bayesian strategy}
\label{subsec:basic-inequality}

Our first theorem of this paper bounds $\cps_n^\pi$ from below as follows.
\begin{theorem}
\label{thm:basic-inequality}
In OUFG, for any $C \in (0,\min(\epsilon_\pi,1/2))$ and 
$0 < S_n / \ss_n^2 < C/2$,
\begin{equation}
\label{eq:basic-inequality}
\cps_n^{\pi} \ge \frac{\sqrt{C}}{6} \frac{S_n}{\ss_n^2} \pi\left( \frac{S_n}{\ss_n^2}\right)
\exp\left( \frac{(1-2C) S_n^2}{2 \ss_n^2}\right).
\end{equation}
\end{theorem}


\begin{proof}
Let $\delta=C/(1+C)$. Then $\delta < C < \epsilon_\pi$. 
Also
\[
\frac{1+\sqrt{C}}{1+C} \frac{S_n}{\ss_n^2} <  \frac{1+\sqrt{C}}{2} \, \frac{C}{1+C} <  \frac{C}{1+C} = \delta < \epsilon_\pi.
\]
Then, by Lemma \ref{lem:log-inequality}
\allowdisplaybreaks
\begin{align}
\cps_n^{\pi} & 
\ge \int_0^\delta \exp\left(\sum_{i=1}^n \ln (1+\epsilon x_i)\right) 
\pi(\epsilon) d\epsilon \nonumber\\
&\ge \int_0^\delta \exp \left( \epsilon S_n - \frac{1+C}{2}\epsilon^2 \ss_n^2\right) \pi(\epsilon)d\epsilon \nonumber\\
&
= \exp\left( \frac{S_n^2}{2(1+C)\ss_n^2}\right)
\int_0^\delta \exp\left( - \frac{(1+C)\ss_n^2}{2} \left(\epsilon- \frac{S_n}{(1+C)\ss_n^2}\right)^2 \right)
\pi(\epsilon) d\epsilon \nonumber\\
& \ge 
\exp\left( \frac{S_n^2}{2(1+C)\ss_n^2}\right)
\int_{\frac{S_n}{\ss_n^2}}^{\frac{1+\sqrt{C}}{1+C}\frac{S_n}{\ss_n^2}} \exp\left( - \frac{(1+C)\ss_n^2}{2} 
\left(\epsilon- \frac{S_n}{(1+C)\ss_n^2}\right)^2 \right)\frac{1}{\epsilon} \epsilon\pi(\epsilon) d\epsilon 
\label{eq:sqrt-c-upper-bound}
\\
& \ge
\exp\left( \frac{S_n^2}{2(1+C)\ss_n^2}\right) \,
\frac{(1+C)\ss_n^2}{(1+\sqrt{C})S_n} \, \frac{S_n}{\ss_n^2} \pi\Big(\frac{S_n}{\ss_n^2}\Big)
\int_{\frac{S_n}{\ss_n^2}}^{\frac{1+\sqrt{C}}{1+C}\frac{S_n}{\ss_n^2}} 
\exp\left( - \frac{(1+C)\ss_n^2}{2} 
\left(\epsilon- \frac{S_n}{(1+C)\ss_n^2}\right)^2 \right) d\epsilon \nonumber\\
&\ge \exp\left( \frac{S_n^2}{2(1+C)\ss_n^2}\right)  \frac{1+C}{1+\sqrt{C}}  
\pi\Big(\frac{S_n}{\ss_n^2}\Big) \, \frac{\sqrt{C}(1-\sqrt{C}) S_n}{(1+C)\ss_n^2} 
\exp\left( - \frac{(1+C)S_n^2}{2\ss_n^2} 
\left(\frac{1+\sqrt{C}}{1+C}-\frac{1}{1+C}\right)^2 \right) \nonumber\\
&= \exp\left( \frac{S_n^2}{2(1+C)\ss_n^2}\right) \, \frac{1-\sqrt{C}}{1+\sqrt{C}} \,  \sqrt{C} \,
\frac{S_n}{\ss_n^2} \pi\Big(\frac{S_n}{\ss_n^2}\Big) \, 
\exp\left( - \frac{CS_n^2}{2(1+C)\ss_n^2} \right) \nonumber\\
&= 
\frac{1-\sqrt{C}}{1+\sqrt{C}} \,  \sqrt{C} \,
\frac{S_n}{\ss_n^2} \pi\Big(\frac{S_n}{\ss_n^2}\Big)  \exp\left( \frac{(1-C)S_n^2}{2(1+C)\ss_n^2}\right).
\nonumber\end{align}
Now 
\[
\frac{1-C}{1+C} > (1-C)^2 > 1-2C
\]
and for $C<1/2$
\[
\frac{1-\sqrt{C}}{1+\sqrt{C}} > \frac{1-\sqrt{1/2}}{1+\sqrt{1/2}} > 0.171 > \frac{1}{6}.
\]
Hence we obtain \eqref{eq:basic-inequality}.
\end{proof}


\subsection{Uniform prior}
\label{subsec:uniform}
Let $\pi(\epsilon)=1$, $\epsilon \in (0,1)$, be the uniform prior.  By Theorem \ref{thm:basic-inequality}
we obtain the following result.

\begin{proposition}
\label{prop:uniform}
In OUFG, by the uniform prior, Skeptic weakly forces
\begin{equation}
\label{eq:uniform-prior}
\ss_\infty = \infty \ \ \Rightarrow \ \ \limsup_n \frac{S_n}{\sqrt{\ss_n^2 \ln \ss_n^2}} \le 1.
\end{equation}
\end{proposition}

\begin{proof}  
Suppose that Reality chooses a path such that $\ss_\infty = \infty$, 
$\limsup_n S_n/ \ss_n^2 \le 0$  and
\[
\limsup_n \frac{S_n}{\sqrt{\ss_n^2 \ln \ss_n^2}} > 1.
\]
Then for some small positive $C$, we have
\[
\limsup_n \frac{S_n}{\sqrt{\ss_n^2 \ln \ss_n^2}} > \frac{1}{\sqrt{1-2C}}
\]
and 
\[
\frac{S_n^2}{\ss_n^2} \ge \frac{\ln \ss_n^2}{1-2C}
\]
for infinitely many $n$.  For such an $n$,
\begin{align*}
\frac{\sqrt{C}}{6} \frac{S_n}{\ss_n^2} \pi\left( \frac{S_n}{\ss_n^2}\right)
\exp\left( \frac{(1-2C) S_n^2}{2 \ss_n^2}\right)
& \ge \frac{\sqrt{C}}{6} \exp\left( \frac{(1-2C) S_n^2}{2 \ss_n^2} + \ln \frac{S_n}{\ss_n^2} \right) \\
& \ge \frac{\sqrt{C}}{6} \exp\left( \frac{\ln \ss_n^2}{2} + \ln \frac{S_n}{\ss_n} + \ln \frac{1}{\ss_n}\right) \\
&= \frac{\sqrt{C}}{6} \exp\left( \ln \frac{S_n}{\ss_n}\right) = \frac{\sqrt{C}}{6} \frac{S_n}{\ss_n} \\
& \ge \frac{\sqrt{C}}{6} \frac{1}{\sqrt{1-2C}} \sqrt{\ln \ss_n^2}.
\end{align*}
Since $\lim_n \ss_n = \infty$, we have $\limsup_n \cps_n^\pi = \infty$.
\end{proof}

In Proposition \ref{prop:uniform}, we considered continuous uniform mixture.
However the same proof can be applied to discrete mixture setting.  In particular the 
discrete mixture in the basic proof of SLLN in Chapter 3 of \cite{ShaferVovk2001Probability}
achieves the same rate $\sqrt{\ss_n^2 \ln \ss_n^2}$ of \eqref{eq:uniform-prior}.
Note also that this rate is the same
as the rate of SLLN of dynamic strategies studied in 
\cite{KumonTakemuraTakeuchi2011Sequential} and
\cite{KumonTakemura2008AISM}.

\subsection{A refined lower bound of the capital process}
\label{subsec:another-bound}

By Assumption \ref{ass:1}, $\pi(\epsilon)\ge \delta_\pi > 0$ holds around the
origin. This implies that for any prior satisfying Assumption \ref{ass:1}, the
rate of convergence $S_n=O(\sqrt{\ss_n^2 \ln \ss_n^2})$ already holds.  
In fact, for $\epsilon<\epsilon_\pi$ write 
\[
\pi(\epsilon) =  \delta_\pi + (\pi(\epsilon)-\delta_\pi).
\]
Then 
\[
\cps_n^\pi \ge \delta_\pi \int_0^{\epsilon_\pi} \prod_{i=1}^n (1+\epsilon x_i) d\epsilon
\]
and the result for the uniform prior applies to the right-hand side.
Hence in considering other priors satisfying  Assumption \ref{ass:1}, we can only 
consider paths such that $\limsup_n S_n /\sqrt{\ss_n^2 \ln \ss_n^2} \le 1$.  In particular
$S_n^3 = O((\ss_n^2 \ln \ss_n^2)^{3/2})$ and $S_n^3/\ss_n^4 \rightarrow 0$ as $n\rightarrow\infty$.
For such paths, we consider  maximizing the right-hand side of \eqref{eq:basic-inequality}  with respect to 
$C$. For the case $S_n/\ss_n$ is large, we put $C=\ss_n^2/S_n^2$ and obtain the following theorem.

\begin{theorem}
\label{thm:basic-inequality-2}
In OUFG, if $S_n>0$, $S_n^2 / \ss_n^2 > \max(2,1/\epsilon_\pi)$ and $S_n^3/\ss_n^4 <1/2$, then
\begin{equation}
\label{eq:basic-inequality-1}
\cps_n^{\pi} \ge  
\frac{1}{6e} \frac{1}{\ss_n}
\pi\left( \frac{S_n}{\ss_n^2}\right)
\exp\left(\frac{S_n^2}{2\ss_n^2}\right).
\end{equation}
\end{theorem}

\begin{proof}
Let $C=\ss_n^2/S_n^2.$  By the assumption $C<\min(1/2,\epsilon_\pi)$.   Also
\[
\frac{S_n}{\ss_n^2} < \frac{\ss_n^2}{2S_n^2} = \frac{C}{2}. 
\]
Hence the conditions of Theorem \ref{thm:basic-inequality} are satisfied.  
Furthermore with this $C$
\[
\exp\left( \frac{(1-2C)S_n^2}{2\ss_n^2}\right)
=\exp\left( \frac{S_n^2}{2\ss_n^2}-1\right).
\]
Substituting the above inequalities into 
\eqref{eq:basic-inequality} we obtain \eqref{eq:basic-inequality-1}.
\end{proof}

\begin{remark}
\label{rem:upper-limit}
In Section \ref{sec:equivalence} we will multiply $\pi(\epsilon)$ by a
decreasing function $c(\epsilon)$ and use $c(\epsilon) \pi(\epsilon)$ as the prior.
Then $\epsilon c(\epsilon) \pi(\epsilon)$ may not be increasing.  
However in view of the range of integration in \eqref{eq:sqrt-c-upper-bound}, we can still 
bound the capital process by
\begin{equation}
\label{eq:bound-for-general-efkp}
\cps_n^{\pi} \ge   c(u_n)
\frac{1}{6e} \frac{1}{\ss_n}
\pi\left( \frac{S_n}{\ss_n^2}\right)
\exp\left( \frac{S_n^2}{2\ss_n^2}\right), \qquad
u_n = \frac{1+\ss_n/S_n}{1+\ss_n^2/S_n^2} \frac{S_n}{\ss_n^2}.
\end{equation}
\end{remark}

For the rest of this section we apply Theorem \ref{thm:basic-inequality-2} to some typical 
nonuniform priors.

\subsection{Power prior}
\label{subsec:power-prior}
For $0 < a < 1$, let
\begin{equation}
\label{eq:power-prior}
\pi(\epsilon) = \epsilon^{-a}, \quad \epsilon\in (0,1).
\end{equation}
For this power prior, the following result holds.
\begin{proposition}
\label{prop:power}
In OUFG, by the power  prior, Skeptic weakly forces
\[
\ss_\infty = \infty \ \ \Rightarrow \ \ \limsup_n \frac{S_n}{\sqrt{(1-a)\ss_n^2\ln \ss_n^2}} \le 1
\]
\end{proposition}

\begin{proof}
Consider the case  $S_n > \sqrt{(1+\delta) (1-a)\ss_n^2 \ln \ss_n^2}$ for some $\delta>0$
for infinitely many $n$.
Note that the function $f(x)=x^{-a/2}e^{x/2}$ is increasing 
for sufficiently large $x$. 
Then for sufficiently large $n$,
\begin{align*}
\frac{1}{\ss_n}
\left( \frac{S_n}{\ss_n^2}\right)^{-a}
\exp\left( \frac{S_n^2}{2 \ss_n^2}\right)
&
= \ss_n^{-(1-a)} \left( \frac{S_n}{\ss_n}\right)^{-a} \exp\left( \frac{S_n^2}{2 \ss_n^2}\right)\\
&\ge \ss_n^{-(1-a)}\, \big((1+\delta) (1-a)\ln \ss_n^2\big)^{-a/2} \, 
\exp\left(\frac{(1+\delta)(1-a)}{2}\ln \ss_n^2\right)\\
&= \frac{1}{\big((1+\delta) (1-a)\big)^{a/2}}\,  \frac{\ss_n^{\delta(1-a)}}{(\ln \ss_n^2)^{a/2}} 
\; \uparrow\infty  \quad (n\rightarrow\infty).
\end{align*}
\end{proof}
\subsection{Prior for the validity of LIL}
\label{subsec:prior-for-usual-lil}

Here we present a prior for the validity of LIL.
Let $\epsilon_0>0$ sufficiently small such that $\ln \ln \epsilonr$ is positive for $\epsilon \le \epsilon_0$.
Define
\begin{equation}
\label{eq:prior-for-lil}
\pi(\epsilon)= \frac{1}{\epsilon \, \ln \epsilonr \, \big(\ln \ln  \epsilonr\big)^2}, \quad \epsilon \in (0,\epsilon_0).
\end{equation}
This $\pi$ is integrable around the origin.  

\begin{proposition}
\label{prop:usual-lil}
In OUFG, by the prior \eqref{eq:prior-for-lil}, Skeptic weakly forces
\begin{equation}
\label{eq:usual-lil-prior}
\ss_\infty = \infty \ \ \Rightarrow \ \ \limsup_n \frac{S_n}{\sqrt{2\ss_n^2 \ln \ln \ss_n^2}} \le 1.
\end{equation}
\end{proposition}

\begin{proof}
Again we use the fact that $f(x)=x^{-1/2} e^{x/2}$ is increasing for sufficiently large $x$.
For $S_n > \max(\sqrt{2(1+\delta) \ss_n^2 \ln \ln \ss_n^2},1)$ for some  $\delta>0$ and sufficiently large $n$, 
we have 
\begin{align*}
\frac{1}{\ss_n}  \pi\left(\frac{S_n}{\ss_n^2}\right) \exp\left( \frac{S_n^2}{2 \ss_n^2}\right)
&=  \frac{S_n}{\ss_n^2}\pi\left(\frac{S_n}{\ss_n^2}\right) \;
\left(\frac{S_n}{\ss_n}\right)^{-1} \exp\left( \frac{S_n^2}{2 \ss_n^2}\right) \\
&=   \frac{1}{ \ln \frac{\ss_n^2}{S_n} \, \left(\ln \ln \frac{\ss_n^2}{S_n}\right)^2} \,
\left( \frac{S_n}{\ss_n}\right)^{-1}\exp\left( \frac{S_n^2}{2 \ss_n^2}\right)\\
&\ge \frac{1}{\big(2(1+\delta)\ln\ln \ss_n^2\big)^{1/2}}
\exp\left( (1+\delta) \ln \ln \ss_n^2  - \ln \ln \frac{\ss_n^2}{S_n} -2\ln \ln\ln \frac{\ss_n^2}{S_n} \right) \\
& \ge \frac{1}{\big(2(1+\delta)\big)^{1/2}} \exp\left( (1+\delta) \ln \ln \ss_n^2  - \ln \ln \ss_n^2 
- \frac{5}{2}\ln \ln\ln \ss_n^2  \right) \\
& = \frac{1}{\big(2(1+\delta)\big)^{1/2}} \exp\left( \delta \ln \ln \ss_n^2
 - \frac{5}{2} \ln \ln\ln \ss_n^2  \right) 
\; \uparrow\infty  \quad (n\rightarrow\infty).
\end{align*}
\end{proof}


\subsection{Priors for the validity of typical EFKP-LIL}
\label{subsec:prior-for-typical-EFKP-LIL}

Here we generalize the prior \eqref{eq:usual-lil-prior} in view of EFKP-LIL.
We give a more general and complete treatment of EFKP-LIL in Section
\ref{sec:equivalence}.

Write
\[
\ln_b x = \underbrace{\ln \ln \dots \ln}_{b\  \text{times}} x.
\]
The following prior density
\begin{equation}
\label{eq:prior-for-EFKP-LIL}
\pi(\epsilon) = \frac{1}{\epsilon \big(\ln \epsilonr\big) \, \big(\ln_2 \epsilonr\big)\cdots
\big(\ln_{b-1}\epsilonr\big) \big(\ln_b\epsilonr\big)^{1+\gamma}},  \qquad \gamma>0
\end{equation}
is integrable around the origin.   We compare this with the bound $\ss_n \psi(\ss_n^2)$, where
\begin{equation}
\label{eq:bound-for-EFKP-LIL}
\psi(\ss_n^2)=\sqrt{2 \ln_2 \ss_n^2 + 3 \ln_3 \ss_n^2+ 2 \ln_4 \ss_n^2+ \dots + 2 \ln_b \ss_n^2 + 2(1+2\gamma) \ln_{b+1} \ss_n^2}.
\end{equation}

For notational simplicity we take $b\ge 4$ in \eqref{eq:prior-for-EFKP-LIL}, since the coefficient
of  $\ln_k$ is 2 except for  $k=3$ in \eqref{eq:bound-for-EFKP-LIL}.
Note that $\gamma$ in \eqref{eq:prior-for-EFKP-LIL} is multiplied by 2 
in \eqref{eq:bound-for-EFKP-LIL}.
This is needed, because unlike the usual LIL in \eqref{eq:usual-lil-prior}, where we considered the ratio of $S_n$ to $\sqrt{2 \ss_n^2 \ln \ln \ss_n^2}$, in EFKP-LIL we need to consider the difference $S_n - \ss_n \psi(\ss_n^2)$.  This difference corresponds to multiplication of $\pi(\epsilon)$ by $c(\epsilon)$
in the proof of Theorem \ref{thm:efkp-general-pi}.

\begin{proposition}
\label{prop:efkp-lil}
In OUFG, by the prior \eqref{eq:prior-for-EFKP-LIL}, Skeptic weakly forces
\[
\ss_\infty = \infty \ \ \Rightarrow \ \ 
S_n - \ss_n \psi(\ss_n^2) 
\le 0 \quad {a.a.},
\]
where $\psi(\ss_n^2)$ is defined in \eqref{eq:bound-for-EFKP-LIL}.
\end{proposition}

\begin{proof}  
For  $S_n > \max(\ss_n\psi(\ss_n^2),1)$
and sufficiently large $n$, we have
\[
\left( \frac{S_n^2}{\ss_n^2}\right)^{-1/2}\exp\left( \frac{S_n^2}{2 \ss_n^2}\right)  \ge 
\frac{1}{\psi(\ss_n^2)} \exp\left(\frac{\psi(\ss_n^2)^2}{2}\right)
= \exp\left(\frac{\psi(\ss_n^2)^2}{2} - \ln\psi(\ss_n^2)\right).
\]
Here
\begin{align*}
\ln\psi(\ss_n^2) &= \frac{1}{2}\ln (2\ln_2 \ss_n^2) 
+ \frac{1}{2} \ln \left(1+ \frac{3 \ln_3 \ss_n^2+ 2 \ln_4 \ss_n^2+ \dots + 2 \ln_b \ss_n^2 + 2(1+2\gamma) \ln_{b+1} \ss_n^2}{\ln_2 \ss_n^2} \right) \\
&=\frac{1}{2}\ln_3 \ss_n^2 + \frac{1}{2}\ln 2 + o(1).
\end{align*}
Combining this with
\begin{align*}
\ln \left[\frac{S_n}{\ss_n^2} \pi\left( \frac{S_n}{\ss_n^2}\right)\right]
&= - \ln_2 \frac{\ss_n^2}{S_n} - \ln_3 \frac{\ss_n^2}{S_n}  - \dots - \ln_b \frac{\ss_n^2}{S_n} -
(1+\gamma)  \ln_{b+1} \frac{\ss_n^2}{S_n}\\
&\ge - \ln_2 \ss_n^2 - \ln_3 \ss_n^2  - \dots - \ln_b \ss_n^2 - 
(1+\gamma)  \ln_{b+1} \ss_n^2,
\end{align*}
we have
\[
\frac{1}{\ss_n}
\pi\left( \frac{S_n}{\ss_n^2}\right)
\exp\left( \frac{S_n^2}{2 \ss_n^2}\right)\ge
\exp\left( \gamma \ln_{b+1} \ss_n^2 - \frac{1}{2}\ln 2 - o(1) \right)
\; \uparrow\infty  \quad (n\rightarrow\infty).
\]
\end{proof}

\section{Equivalence of prior densities and the upper class}
\label{sec:equivalence}

The typical priors of the previous section suggest that higher concentration of the prior
around the origin corresponds to a tighter convergence rate of SLLN.
In particular, from the viewpoint of EFKP-LIL, it is of interest to establish the relation
between priors and the function of the upper class.  

Also it is natural to conjecture that the rate of SLLN implied by a Bayesian strategy
only depends on the rate of concentration of the prior around the origin.  This idea can be 
clarified if we classify priors with the same rate of concentration into the same class.
Let $\Pizero$ denote the set of priors $\pi$ satisfying Assumption \ref{ass:1} by
\begin{equation}
\Pizero = \{ \pi \mid \pi \ \text{satisfies Assumption \ref{ass:1}} \}.
\end{equation}

\subsection{The upper class}

In studying the validity of EFKP-LIL, the notion of upper class of functions is essential.
A positive function $\psi(\lambda)$ defined for $\lambda > M_\psi >0$ is said to \emph{belong to the upper class in OUFG} if
Skeptic can force the event
\[A_\infty=\infty\Rightarrow S_n-\ss_n\psi(\ss_n^2)\le0 \quad {a.a.}\]
For the terminology, see \cite{Revesz2013Random}.
We characterize the upper class by an integral test:
\begin{equation}
\label{eq:upper-class}
\int_{M_\psi}^\infty \psi(\lambda) e^{-\psi(\lambda)^2/2}
\frac{d\lambda}{\lambda} < \infty.
\end{equation}
A typical function in the upper class is $\psi$ given in \eqref{eq:bound-for-EFKP-LIL}.
For convenience, we put the following regularity conditions on $\psi$.

\begin{assumption}
\label{ass:2}
There exist some $M_\psi > 0$ and $\delta_\psi > 0$, such that
\begin{enumerate}
\item $\psi(\lambda)$ is a positive increasing function on $(M_\lambda,\infty)$,
\item the integral in  \eqref{eq:upper-class} is finite, and
\item for $\lambda>M_\psi$, we have
\begin{equation}
\label{eq:ass:2}
\lambda \psi(\lambda) e^{ -\psi(\lambda)^2/2}  > \delta_\psi.
\end{equation}
\end{enumerate}
\end{assumption}

If $\lim_{\lambda\uparrow\infty} \psi(\lambda)$ is finite, then the integral in
\eqref{eq:upper-class} diverges, since $\int_{M_\psi}^\infty d\lambda/\lambda = \infty$.
Hence $\lim_{\lambda\uparrow\infty} \psi(\lambda)=\infty$ for $\psi$ satisfying 
Assumption \ref{ass:2}. Furthermore the function $x e^{-x^2/2}$ is 
decreasing for $x>1$ and converges to zero as $x\rightarrow\infty$.  
Hence $\psi(\lambda) e^{ -\psi(\lambda)^2/2}$ is eventually 
decreasing to zero.  For EFKP-LIL, as the typical $\psi$ in \eqref{eq:bound-for-EFKP-LIL},
we are mainly concerned with functions $\psi$, for which 
$\psi(\lambda) e^{ -\psi(\lambda)^2/2}$ decreases to zero 
slower than any negative power of $\lambda$. Hence 
\eqref{eq:ass:2} is a mild regularity condition.
We denote the set of functions $\psi$ satisfying Assumption \ref{ass:2} by
\[
\Psiinfty = \{ \psi \mid \psi \ \text{satisfies Assumption \ref{ass:2}} \}.
\]

The goal is, when a $\psi\in\Psiinfty$ is given, to find a $\pi\in\Pizero$
that is a witness of the $\psi$ in the upper class.
Now we define two functionals $F:\Psiinfty \rightarrow \Pizero$ and 
$G:\Pizero \rightarrow \Psiinfty$ as follows.
\begin{definition}
\begin{align}
F[\psi](\epsilon) &= \frac{\psi(\epsilonr)}{\epsilon}  \exp\left(-\psi\left(\epsilonr\right)^2/2\right), 
\label{eq:def-F}
\\
G[\pi](\lambda)&= \sqrt{\beta\left(\frac{1}{\lambda}\right) + \ln \beta\left(\frac{1}{\lambda}\right)}, \qquad \beta(\epsilon)=\max(-2 \ln (\epsilon\pi(\epsilon)),1).
\label{eq:def-G}
\end{align}
\end{definition}
\begin{figure}[thbp]
\begin{center}
\includegraphics[width=11cm]{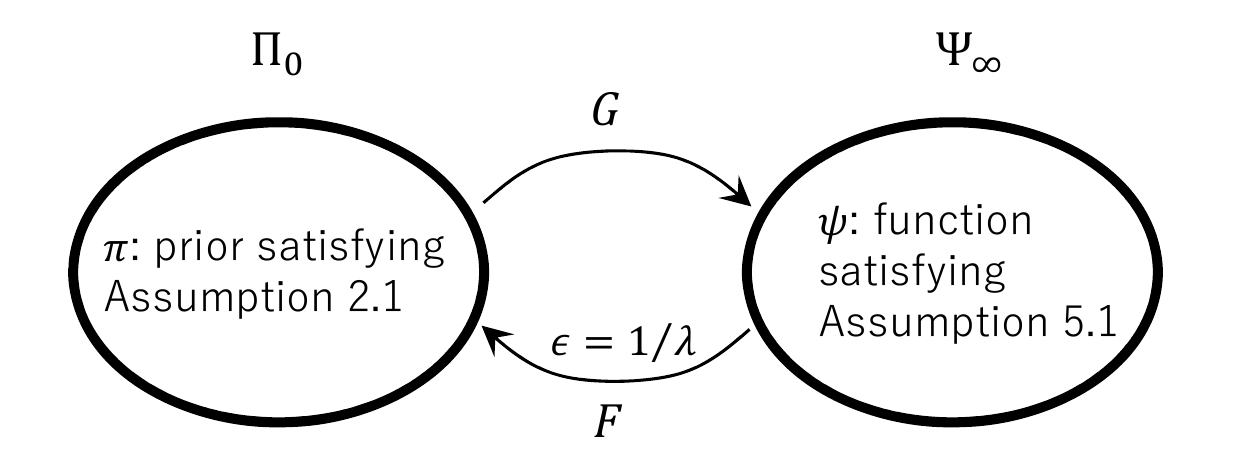}
\caption{Functionals $F$ and $G$}
\label{fig:1}
\end{center}
\end{figure}
See Figure \ref{fig:1}. 
These two functionals are ``asymptotic inverse functionals'' as we clarify 
in \eqref{eq:asymptotic-inverse} and Theorem \ref{thm:asymptotic-equivalence}.
Since we are interested in the case that $\epsilon$ is sufficiently small and $\lambda$ is sufficiently large, for the rest of this paper we only consider $\lambda$ that fulfills the condition $-2\ln(\frac{1}{\lambda}\pi(\frac{1}{\lambda})) \ge 1$. In fact, this condition holds true for
sufficiently large $\lambda$ because of the monotonicity of $\frac{1}{\lambda}\pi(\frac{1}{\lambda})$ and $\lim_{\lambda\uparrow\infty}\frac{1}{\lambda}\pi(\frac{1}{\lambda})=0$.
First we need to check that $F[\psi]\in \Pizero$ and $G[\pi]\in \Psiinfty$.
For checking $G[\pi]\in \Psiinfty$, we use the following relation.
\begin{align}
G[\pi](\lambda)\; \exp\left( - \frac{(G[\pi](\lambda))^2}{2}\right)
&=\sqrt{\beta(1/\lambda) + \ln \beta(1/\lambda)} \;
\frac{\exp\left(- \beta(1/\lambda)/2\right)}{\sqrt{\beta(1/\lambda)}}
\nonumber \\
&= \sqrt{1+\frac{\ln\beta(1/\lambda)}{\beta(1/\lambda)}} 
\lambdar \pi\left(\lambdar\right).
\label{eq:G-pi-exp-square}
\end{align}

\begin{lemma}
\label{lem:F-G}
$F[\psi]\in \Pizero$ for  $\psi\in \Psiinfty$.  $G[\pi]\in \Psiinfty$ for $\pi\in\Pizero$.
\end{lemma}

\begin{proof}
For the first part, the monotonicity of $\epsilon F[\psi](\epsilon) \in \Pizero$ for
sufficiently small $\epsilon$ holds because  $\psi(\lambda) e^{ -\psi(\lambda)^2/2}$ is eventually 
decreasing.  The integrability of  $F[\psi]$ follows directly from the
change of variables $\epsilon=1/\lambda$.  Finally $F[\psi](\epsilon) > \delta_\psi$
for $\epsilon < 1/M_\psi$ by 
\eqref{eq:ass:2} and \eqref{eq:def-F}.

For the second part, the monotonicity of $G[\pi]$ is obvious.
We note the fact that $\lim_{\epsilon\downarrow 0}\epsilon\pi(\epsilon) =0$,
because otherwise $\pi(\epsilon)$ is not integrable around the origin.
Hence 
\begin{equation}
\label{eq:beta-zero}
\beta(\epsilon)\uparrow\infty \ \ \text{as} \ \ \epsilon\downarrow 0.
\end{equation}
Then $G[\pi](\lambda) = \sqrt{\beta(1/\lambda)}(1+o(1))$ as $\lambda\rightarrow\infty$.
Then as $\lambda\rightarrow\infty$, by \eqref{eq:G-pi-exp-square},
\begin{equation}
\label{eq:G-pi-integrable}
G[\pi](\lambda)\; \exp\left( - \frac{(G[\pi](\lambda))^2}{2}\right) \frac{1}{\lambda}=
 \sqrt{1+\frac{\ln\beta(1/\lambda)}{\beta(1/\lambda)}} 
\frac{1}{\lambda^2} \pi\left(\lambdar\right)
=
\frac{1}{\lambda^2} \pi\left(\lambdar\right)
(1+o(1)).
\end{equation}
By the change of variables $\epsilon=1/\lambda$, the integrability of the left-hand 
side of \eqref{eq:G-pi-integrable} reduces to the integrability of 
$\pi(\epsilon)$ around the origin.
Finally 
\[
\lambda G[\pi](\lambda)\; \exp\left( - \frac{(G[\pi](\lambda))^2}{2}\right) 
= \pi\left(\lambdar\right)
(1+o(1)),
\]
which is bounded away from zero for large $\lambda$.
\end{proof}

Based on this lemma, we can consider the compositions $F\circ G: \Pizero \rightarrow \Pizero$ and
$G\circ F: \Psiinfty \rightarrow\Psiinfty$.  By direct computation we obtain
\begin{align}
F[G[\pi]](\epsilon) 
&= \pi(\epsilon) \sqrt{\frac{\beta(\epsilon) + \ln \beta(\epsilon)}{\beta(\epsilon)}}, \qquad
\beta(\epsilon)=-2 \ln(\epsilon\pi(\epsilon)),
\label{eq:FG-composition} \\
G[F[\psi]](\lambda) &= \sqrt{\psi(\lambda)^2 - 2\ln \psi(\lambda) + \ln\left(\psi(\lambda)^2 - 2\ln \psi(\lambda)\right)}.
\label{eq:GF-composition}
\end{align}

\subsection{Validity of EFKP-LIL via Bayesian strategy}
\label{subsec:efkp-validity}
Here
we establish the validity of EFKP-LIL via a Bayesian strategy.

\begin{theorem}
\label{thm:efkp-general-pi}
In OUFG, by a prior $\pi \in \Pizero$, Skeptic weakly forces
\[
\ss_\infty = \infty \ \ \Rightarrow \  \ 
S_n - \ss_n G[\pi](\ss_n^2) \le 0 \quad {a.a.}
\]
\end{theorem}

\begin{proof}
Given $\pi\in \Pizero$, there exists a decreasing function $c(\epsilon)$, such that 
$\lim_{\epsilon\downarrow 0}c(\epsilon)=\infty$ and
$\int_0^1 c(\epsilon) \pi(\epsilon) < \infty$.  Indeed we can define
\[
c(\epsilon)=k \ \ \text{for} \ \ \epsilon\in [\epsilon_{k+1},\epsilon_k),  \ k=1,2,\dots,
\]
where $\epsilon_k$ is defined by 
\[
\int_0^{\epsilon_k} \pi(\epsilon) d\epsilon =  \frac{1}{2^k}\int_0^1 \pi(\epsilon) d\epsilon.
\]
We use the Bayesian strategy with the prior $\tilde \pi(\epsilon)=c(\epsilon)\pi(\epsilon)$.
Write $\psi(\lambda)=G[\pi](\lambda)$.
For sufficiently large $n$ such that
$S_n \ge \max(\ss_n \psi(\ss_n^2),1)$ and 
$S_n /\ss_n^2$ is sufficiently small,  by \eqref{eq:bound-for-general-efkp} in Remark \ref{rem:upper-limit}, we have
\begin{align*}
\cps_n^{\pi} 
&\ge c(u_n)\frac{1}{\ss_n} \pi\left(  \frac{S_n}{\ss_n^2} \right) 
 \exp\left(\frac{S_n^2}{2\ss_n^2}\right)  \\
&=
c(u_n)\frac{S_n}{\ss_n^2} \pi\left(  \frac{S_n}{\ss_n^2} \right) 
\left( \frac{S_n^2}{\ss_n^2}\right)^{-1/2} \exp\left(\frac{S_n^2}{2\ss_n^2}\right) \\
&\ge  c(u_n) \frac{S_n}{\ss_n^2} \pi\left( \frac{S_n}{\ss_n^2} \right) 
\exp\left( \frac{\psi_n(\ss_n^2)^2}{2}  - \ln \psi(\ss_n^2)\right) \\
& \ge  c(u_n) \frac{1}{\ss_n^2} \pi\left(\frac{1}{\ss_n^2}\right) 
\exp\left( \frac{\psi(\ss_n^2)^2}{2} - \ln \psi(\ss_n^2)\right) .
\end{align*}
Now 
\[
\exp\left( \frac{\psi(\ss_n^2)^2}{2}\right)
= \exp\left( \frac{1}{2} \beta\left(\frac{1}{\ss_n^2}\right)
+ \frac{1}{2}\ln \beta\left(\frac{1}{\ss_n^2}\right)\right) 
=\left(\frac{1}{\ss_n^2} \pi\left(\frac{1}{\ss_n^2}\right) \right)^{-1} 
\exp\left(  \frac{1}{2}\ln \beta\left(\frac{1}{\ss_n^2}\right)\right) .
\]
Hence
\[
\cps_n^{\pi} \ge c(u_n) \exp\left( \frac{1}{2}\ln \beta\left(\frac{1}{\ss_n^2}\right) -\ln \psi(\ss_n^2) \right).
\]
But by \eqref{eq:def-G},
\[
\ln \psi(\ss_n^2)  = \frac{1}{2}\ln \beta\left( \frac{1}{\ss_n^2}\right)  +o(1)
\]
and we have
\[
\cps_n^{\pi} \ge c(u_n) \exp(o(1)) \ \uparrow \infty  \qquad (n\rightarrow\infty),
\]
because $u_n \rightarrow 0$ as $n\rightarrow\infty$.
\end{proof}

As a corollary to this theorem, we state the following statement of the validity of EFKP-LIL.
\begin{corollary}
Let $\psi$ be a positive increasing function defined for $\lambda>M_{\psi}>0$.
Then, $\psi$ belongs to the upper class in OUFG
if and only if 
\[I(\psi)=\int_{M_{\psi}}^\infty \psi(\lambda)e^{-\psi(\lambda)^2/2}\frac{d\lambda}{\lambda}<\infty.\]
\end{corollary}

\begin{proof}
First suppose that $\psi\in\Psiinfty$.
Let $\pi=F[\psi]$ and apply Theorem 
\ref{thm:efkp-general-pi}.  Then Skeptic can weakly force
\[
\ss_\infty = \infty \ \ \Rightarrow \  \ 
S_n - \ss_n G[F[\psi]](\ss_n^2) \le 0 \quad {a.a.}
\]
But in \eqref{eq:GF-composition} we have
\[
G[F[\psi]](\lambda) < \psi(\lambda)
\]
for all sufficiently large $\lambda$.
Hence $S_n - \ss_n\psi(\ss_n^2) \le 0 \quad {a.a.}$

Next let $\psi$ be a positive increasing function such that $I(\psi)<\infty$.
Let $\widetilde{\psi}=\min\{\psi,\sqrt{2\ln_2 n+4\ln_3 n}\}$.
Notice that $\sqrt{2\ln_2 n+4\ln_3 n}$ belongs to the upper class.
Since $xe^{-x^2/2}$ is decreasing for $x>1$ and $\widetilde{\psi}\le\sqrt{2\ln_2 n+4\ln_3 n}$,
we have
\[\lambda\widetilde{\psi}(\lambda)e^{-\widetilde{\psi}(\lambda)^2/2}
\ge\lambda\frac{\sqrt{2\ln_2 n+4\ln_3 n}}{(\ln n)(\ln_2 n)^2}.\]
Then, $\widetilde{\psi}$ satisfies Assumption \ref{ass:2}.
Hence, $\widetilde{\psi}$ belongs to the upper class and so does $\psi$.

Finally, suppose that $\psi$ is a positive increasing function,
but $I(\psi)=\infty$.
Then, in the fair-coin tossing game,
Skeptic can force $S_n-\sqrt{n}\psi(n)\ge0$ for infinitely many $n$.
This is a corollary from the classical EFKP-LIL.
This fact also follows from the main theorem in \cite{sasai2015efkp-sn}.
Thus, Reality can comply with this event with this restriction,
which means that $\psi$ does not belong to the upper class in OUFG.
\end{proof}

\subsection{Equivalence}

Consider $\psi\in\Psiinfty$ such that the convergence rate of $I(\psi)$ is very slow.
Since $G[F[\psi]]\in\Psiinfty$ and $G[F[\psi]]<\psi$,
the two functions $\psi$ and $G[F[\psi]]$ should have similar growing rates.
One should also notice that, by \eqref{eq:beta-zero},
\begin{equation}
\label{eq:asymptotic-inverse}
\lim_{\epsilon\downarrow 0} \frac{F[G[\pi]](\epsilon)}{\pi(\epsilon)}=1.
\end{equation}
In this asymptotic sense, $F$ and $G$ are inverse functionals.  
We clarify this point further below in Theorem \ref{thm:asymptotic-equivalence}.

In order to identify functions with the same rate of growth,
we introduce equivalence relations in $\Pizero$ and $\Psiinfty$.
For two functions $\pi_1, \pi_2 \in\Pizero$ we write  $\pi_1 \simzero \pi_2$ if
\[
0 < \liminf_{\epsilon\downarrow 0} \frac{\pi_2(\epsilon)}{\pi_1(\epsilon)} \le 
\limsup_{\epsilon\downarrow 0} \frac{\pi_2(\epsilon)}{\pi_1(\epsilon)} < \infty.
\]
It is easy to check that $\simzero$ is an equivalence relation.

For two functions $\psi_1, \psi_2 \in\Psiinfty$ we write  $\psi_1 \siminfty \psi_2$ if
\[
0 < \liminf_{\lambda\uparrow \infty} \frac{\exp(\psi_2(\lambda)^2/2)}{\exp(\psi_1(\lambda)^2/2)}
\le 
\limsup_{\lambda\uparrow \infty}  \frac{\exp(\psi_2(\lambda)^2/2)}{\exp(\psi_1(\lambda)^2/2)}
< \infty.
\]
Again it is easy to check that $\siminfty$ is an equivalence relation.

In the following two lemmas we prove that the functionals $F$ and $G$ preserve the equivalence relations.

\begin{lemma}
\label{lem:F-preserve}
Let $\psi_1, \psi_2 \in \Psiinfty$. If $\psi_1 \siminfty \psi_2$, then
$F[\psi_1] \simzero F[\psi_2]$.
\end{lemma}

\begin{proof}
Since $\psi_1 \siminfty \psi_2$, there exist positive $\kappa_1, \kappa_2, \lambda_1$ such that
\[
\kappa_1 \exp(\psi_1(\lambda)^2/2) \le 
\exp(\psi_2(\lambda)^2/2) \le 
\kappa_2 \exp(\psi_1(\lambda)^2/2), \qquad
\forall \lambda> \lambda_1.
\]
Then 
\[
\sqrt{2 \ln \kappa_1 + \psi_1(\lambda)} \le \psi_2(\lambda) \le \sqrt{2 \ln \kappa_2 + \psi_1(\lambda)}.
\]
Since $\psi_1(\lambda)$ and $\psi_2(\lambda)$ diverge to $\infty$ as $\lambda \rightarrow\infty$,
for sufficiently small $\epsilon$  we have
\[
\frac{\kappa_1}{2}\psi_1\left(\epsilonr\right) \exp\left(\psi_1\left(\epsilonr\right)^2/2\right)
\le \psi_2\left(\epsilonr\right) \exp\left(\psi_2\left(\epsilonr\right)^2/2\right)
\le 2 \kappa_2\psi_1\left(\epsilonr\right) \exp\left(\psi_1\left(\epsilonr\right)^2/2\right).
\]
Dividing this by $\epsilon$ we obtain the lemma.
\end{proof}

\begin{lemma}
\label{lem:G-preserve}
Let $\pi_1, \pi_2 \in \Pizero$. If $\pi_1 \simzero \pi_2$, then
$G[\pi_1] \siminfty G[\pi_2]$.
\end{lemma}

\begin{proof}
There exist positive $\kappa_1, \kappa_2, \epsilon_1$ such that
\begin{equation}
\label{eq:G-preserve-1}
\kappa_1 \pi_1(\epsilon) < \pi_2(\epsilon) < \kappa_2 \pi_1(\epsilon), \qquad 
\forall \epsilon \in (0,\epsilon_1).
\end{equation}
Then 
\[
-2 \ln (\epsilon \pi_1(\epsilon)) - 2 \ln \kappa_2 
\le
-2 \ln (\epsilon \pi_2(\epsilon))
\le 
-2 \ln (\epsilon \pi_1(\epsilon)) - 2 \ln \kappa_1 .
\]
Since $-\ln(\epsilon \pi_i(\epsilon))$, $i=1,2$ diverge to $\infty$ as $\epsilon \downarrow 0$,
for sufficiently small $\epsilon$ we have
\begin{equation}
\label{eq:G-preserve-2}
\frac{1}{2}\sqrt{-2 \ln (\epsilon \pi_1(\epsilon))}
\le 
\sqrt{-2 \ln (\epsilon \pi_2(\epsilon))}
\le 
2\sqrt{-2 \ln (\epsilon \pi_1(\epsilon))}.
\end{equation}
By \eqref{eq:G-preserve-1} and \eqref{eq:G-preserve-2}
\[
\frac{1}{2\kappa_2} \frac{\sqrt{-2 \ln (\epsilon \pi_1(\epsilon))}}{\epsilon \pi_1(\epsilon)}
\le  \frac{\sqrt{-2 \ln (\epsilon \pi_2(\epsilon))}}{\epsilon \pi_2(\epsilon)}
\le 
\frac{2}{\kappa_1} \frac{\sqrt{-2 \ln (\epsilon \pi_1(\epsilon))}}{\epsilon \pi_1(\epsilon)}.
\]
Now
\[
\frac{\sqrt{-2 \ln (\epsilon \pi_i(\epsilon))}}{\epsilon \pi_i(\epsilon)}
= \exp\left( 
\frac{\left(\sqrt{-2 \ln(\epsilon \pi_i(\epsilon)) + \ln (-2 \ln (\epsilon \pi_i(\epsilon))) }\right)^2}
{2}\right) 
= \exp\left(\frac{G[\pi_i]\left(\epsilonr\right)^2}{2}\right), \qquad i=1,2.
\]
Therefore for sufficiently large $\lambda$ 
\[
\frac{1}{2\kappa_2} \exp\left(\frac{G[\pi_1]\left(\lambda\right)^2}{2}\right)
\le
\exp\left(\frac{G[\pi_2]\left(\lambda\right)^2}{2}\right)
\le
\frac{2}{\kappa_1} \exp\left(\frac{G[\pi_1]\left(\lambda\right)^2}{2}\right).
\]
\end{proof}

Let $\Pizero/\simzero$ denote the set of equivalence classes in $\Pizero$ with respect to $\simzero$ and 
define $\Psiinfty /\siminfty$ similarly. Based on the above two lemmas and
\eqref{eq:FG-composition}, \eqref{eq:GF-composition}
we have the following theorem.   
\begin{theorem}
\label{thm:asymptotic-equivalence}
The functionals $F$ and $G$ give bijections between $\Pizero/\simzero$ and  $\Psiinfty /\siminfty$.
Furthermore they are inverse functionals to each other.
\end{theorem}
Proof is obvious and omitted.  
The theorem says that functions in the upper class correspond to the prior densities,
and the integrabilities of $I(\psi)$ correspond to the integrabilities of densities.

\section{Some other properties of the one-sided unbounded forecasting game}
\label{sec:other-properties}
In this paper we have been considering SLLN in the self-normalized form.  However SLLN for OUFG in non-self-normalized form exhibits an interesting property, which we show in the following proposition.
The definition of \emph{compliance} can be found in
\cite{MiyabeTakemura2013Law} and \cite{MiyabeTakemura2015SPA}.

\begin{proposition}
\label{prop:non-sn-slln}
Let ${b_n}$ be a sequence of increasing positive reals such that $\lim_n b_n=\infty$.  In OUFG Skeptic can weakly
force $S_n/b_n \rightarrow 0$ if and only if $\sum_{n} 1/b_n < \infty$.
\end{proposition}

\begin{proof}
Suppose that $Z=\sum_{n} 1/b_n < \infty$. Consider Skeptic's strategy $M_n = 1/b_n$.
Then $Y_n = Z + \sum_{i=1}^n M_i x_i$ is a nonnegative martingale and by the game-theoretic martingale convergence theorem (Lemma 4.5 of \cite{ShaferVovk2001Probability}) Skeptic can weakly force that
$Y_n$ converges to a finite value.  Then by Kronecker's lemma Skeptic can weakly force
$\lim_n \dfrac{S_n}{b_n} = 0$.

The converse part can be proved by a deterministic strategy of Reality complying 
with $\limsup_n |S_n|/b_n \ge 1$.
Suppose $\sum_{n} 1/b_n = \infty$.
If $b_n<n-1$ for infinitely many $n$, then Reality chooses $x_n=-1$ for all $n$,
which complies with the desired property.

Now, we assume that $b_n\ge n-1$ for all but finitely many $n$.
In each round Reality chooses either $x_n = -1$ or $x_n = 2b_n$.
Note that $M_n\ge0$, otherwise,  Skeptic would be bankrupt
when Reality chooses $x_n$ large enough.
Let $p_n = 1/(1+2b_n)$,
\[c_n=\begin{cases}\frac{1}{2(1-p_n)}&\mbox{ if }x_n=-1\\
\frac{1}{2p_n}&\mbox{ if }x_n=2b_n\end{cases}\]
and
\[\cL_n=\cK_n+\prod_{k=1}^n c_k, \qquad \cK_0=1, \ \cL_0=2.\]

We show that, for each $n$, at least one of $x_n\in\{-1,2b_n\}$ satisfies
$\cL_n\le\cL_{n-1}$.
Suppose otherwise.
Then
\[\cK_{n-1}-M_n+\Big(\prod_{k=1}^{n-1}c_k\Big)\cdot\frac{1}{2(1-p_n)}>\cK_{n-1}+
\Big(\prod_{k=1}^{n-1}c_k\Big),\]
\[\cK_{n-1}+2b_nM_n+\Big(\prod_{k=1}^{n-1}c_k\Big)\cdot\frac{1}{2p_n}>\cK_{n-1}+
\Big(\prod_{k=1}^{n-1}c_k\Big).\]
Thus,
\[-(1-p_n)M_n+p_n2b_nM_n>0,\]
which implies
\[p_n>\frac{1}{1+2b_n}.\]
This is a contradiction.

Reality chooses her  move so that $\cL_n\le\cL_{n-1}$ for all $n$.
Then $\sup_n\cK_n\le2$.
Furthermore, $\sup_n\prod_{k=1}^n c_k\le2$.
If $x_n=2b_n$ for at most finitely many $n$,
then there exists $m$ such that $x_n=-1$ for all $n\ge m$,
whence $\prod_{k=1}^n c_k=\prod_{k=1}^{m-1}c_k\cdot\prod_{k=m}^n\frac{1}{2(1-p_n)}\to\infty$ because $\sum_n p_n=\infty$.
Hence, $x_n=2b_n$ for infinitely many $n$.
For such an $n$ that is large enough, we have
\begin{align*}
  S_n=&S_{n-1}+2b_n\ge-(n-1)+2b_n,\\
  \frac{S_n}{b_n}\ge&2-\frac{n-1}{b_n}\ge1,
\end{align*}
which implies $\limsup_n S_n/b_n\ge1$.
\end{proof}

The intuition of $p_n$ is the probability of $x_n=2b_n$
so that the expectation is
$(-1)\cdot(1-p_n)+2b_n p_n=0$.
Then, $\prod_{k=1}^n c_k$ can be seen as a capital process in a game,
and so is $\cL$.

In this proof we constructed a particular deterministic strategy of Reality.
General theory of constructing Reality's deterministic strategy complying with certain events is developed in
\cite{MiyabeTakemura2013Law} and \cite{MiyabeTakemura2015SPA}.

\section{Discussions}
\label{sec:discussions}
In this paper we gave a unified treatment of the rate of convergence of SLLN 
in terms of  Bayesian strategies, including the validity of LIL. 
Concerning LIL we did not discuss the sharpness of the bound.  In OUFG the sharpness
does not hold, because Reality can simply take $x_n \equiv -1$.
For the sharpness we need some boundedness condition, such as the
simplified predictably unbounded forecasting game (Section 5.1 of \cite{ShaferVovk2001Probability}, \cite{sasai2015efkp-sn}).  Even with some boundedness conditions,  current
game-theoretic proofs of the sharpness are still complicated and the nature of the weights, involving also negative ones due to short selling of a strategy, does not seem to be clear.  We hope that the unified treatment of this paper also helps to streamline proofs of the sharpness.

As we discussed in \eqref{eq:portfolio}, a Bayesian strategy  for 
OUFG can be understood as the portfolio of cash and one risky asset.
In the literature on universal portfolio by Thomas Cover and other researchers
(Chapter 16 of \cite{Cover-Thomas-book-2nd}, \cite{Cover-1991MF}, \cite{Ordentlich-Cover-MOR1998}, \cite{Vovk1998Proceedings}),
Bayesian strategies for many risky assets are considered.  
They recommend power priors such as the Dirichlet prior, which corresponds to 
priors in Section \ref{subsec:power-prior}.  From the viewpoint of
the rate of convergence of SLLN, we have shown that we should take $a=1$ in
\eqref{eq:power-prior} with additional multiplicative logarithmic terms.  
Hence our recommendations and those in universal portfolio literature seem to be
different.  This may be due to the difference of criteria for evaluating strategies.
It is interesting to clarify these differences.

\section*{Acknowledgement}\label{sec:acknowlegement}

This research is supported by JSPS Grant-in-Aid for Scientific Research No.\ 16K12399.

\bibliographystyle{abbrv}
\bibliography{bayesian-rate}

\end{document}